\documentclass[oneside,english]{amsart}
\usepackage[T1]{fontenc}
\usepackage[latin9]{inputenc}
\usepackage{amstext}
\usepackage{amsthm}
\usepackage{amssymb}

\makeatletter
\numberwithin{equation}{section}
\numberwithin{figure}{section}
\theoremstyle{plain}
\newtheorem{thm}{\protect\theoremname}[section]
  \theoremstyle{plain}
  \newtheorem{lem}[thm]{\protect\lemmaname}
  \theoremstyle{remark}
  \newtheorem{rem}[thm]{\protect\remarkname}

\@ifundefined{date}{}{\date{}}
\usepackage{algorithm} 
\usepackage{algorithmic} 
\usepackage{dsfont}
\usepackage{pgf,tikz}
\usepackage{mathrsfs}
\usetikzlibrary{arrows}

\makeatother

\usepackage{babel}
  \providecommand{\lemmaname}{Lemma}
  \providecommand{\remarkname}{Remark}
\providecommand{\theoremname}{Theorem}

\begin{document}

\title{A proof of the bunkbed conjecture for the complete graph at $p=\frac{1}{2}$}

\author[P. de Buyer]{Paul de Buyer}
\address[P. de Buyer]{Universit\'e Paris Ouest Nanterre La D\'efense - Modal'X, 200 avenue
de la R\'epublique 92000 Nanterre, France}
\email{debuyer@math.cnrs.fr} 

\keywords{Percolation, Bunkbed graph, Combinatorics, Bunkbed Conjecture}

\subjclass{82B43, 60K35}

\begin{abstract}
The bunkbed of a graph $G$ is the graph $G\times\left\{ 0,1\right\} $.
It has been conjectured that in the independent bond percolation model,
the probability for $\left(u,0\right)$ to be connected with $\left(v,0\right)$
is greater than the probability for $\left(u,0\right)$ to be connected
with $\left(v,1\right)$, for any vertex $u$, $v$ of $G$. In this
article, we prove this conjecture for the complete graph in the case
of the independent bond percolation of parameter $p=1/2$. 
\end{abstract}

\maketitle

\section{Introduction}

Percolation theory has been widely studied over the last decades and
yet, several intuitive results are very hard to prove rigorously.
This is the case of the bunkbed conjecture formulated by Kasteleyn
published as a remark in \cite{van2001correlation}) which investigates
a notion of graph distance through percolation. 

A bunkbed graph of a graph $\widetilde{G}=\left(\widetilde{V},\widetilde{E}\right)$
is the graph $G=\left(V,E\right)=\widetilde{G}\times\left\{ 0,1\right\} $,
to which we have added the edges that connect the vertices $\left(x,0\right)$
to $\left(x,1\right)$ for all vertices $x\in V$, see Figure $\ref{FigureQuiDechire}$.
It is natural to distinguish vertices whether they are on the lower
level, the vertices $\left(x,0\right)$, or on the upper level, the
vertices $\left(x,1\right)$. 

The bunkbed conjecture (see \cite{haggstrom2003probability} for a
more general setting) suggests that two vertices $u=\left(x,0\right)$
and $v=\left(y,0\right)$ on the lower level are closer than $u$
and $v'=\left(y,1\right)$. Closeness of two vertices has to be understood
through the probability of the existence of an open path in the sense
of percolation. 

The percolation model is defined as follow. We open each edges of
$E$ independently with probability $p$ and close them with probability
$1-p$ and we write $\mathbb{P}_{p}$ the law associated to this percolation
model. We call a configuration, an element $\omega=\left(\omega_{e}\right)_{e\in E}\in\left\{ 0,1\right\} ^{E}$
corresponding to the bond percolation model where $\omega_{e}=0$
means that the edge $e$ is closed and $1$ means that the edge $e$
is open. We call an open path a path of open edges and for two vertices
$x,y\in V$, we write $x\leftrightarrow y$ if there exists an open
path between $x$ and $y$. By convention, for any configuration,
a vertex is always connected to itself, i.e. $x\leftrightarrow x$.
For a general introduction on percolation, see \cite{GrimmettPercolation}.

In this article we prove the bunkbed conjecture for the complete graph
when the percolation parameter $p$ is equal to $1/2$.
\begin{thm}
\label{thm:MainTheoremBunkBed-1}Let $G$ be the bunkbed graph of
the complete graph $\widetilde{G}=K_{n}$. For all vertices $x,y$
of $K_{n}$:
\[
\mathbb{P}_{\frac{1}{2}}\left(\left(x,0\right)\leftrightarrow\left(y,0\right)\right)\geq\mathbb{P}_{\frac{1}{2}}\left(\left(x,0\right)\leftrightarrow\left(y,1\right)\right)
\]

\end{thm}
In the previous works of S. Linusson and M. Leander, the conjecture
has been proven for the outerplanar graphs and the wheels graphs using
the so-called method of minimal counter-example, see \cite{leandersjalvstandiga,linusson2011percolation}.
However, this method might not be suitable for the complete graph
due to the its geometric nature. We have chosen to study the conjecture
for the complete graph because we think that it should be easier to
show the following proposition than the bunkbed conjecture itself:
``If the bunkbed conjecture is verified for a graph $\widetilde{G}$,
then it is verified for the graph $\widetilde{G}\backslash\left\{ e\right\} $
where $e$ is an edge of $\widetilde{G}$''. 

One can note that the bunkbed conjecture is true whenever $p$ is
small enough. Indeed, in these cases, only the shortest paths can
be open to connect $\left(x,0\right)$ with $\left(y,\text{0}\right)$
or $\left(y,1\right)$; since the shortest path from $\left(x,0\right)$
to $\left(y,1\right)$ is longer than $\left(x,0\right)$ to $\left(y,0\right)$,
the conjecture is proven. Note $d\left(.,.\right)$ the usual graph
distance, then one can prove that for all vertices $u,v$ and $w$,
$d\left(u,v\right)>d\left(u,w\right)\Rightarrow\mathbb{P}_{p}\left(u\leftrightarrow v\right)>\mathbb{P}_{p}\left(u\leftrightarrow w\right)$
for sufficiently small $p$.

Finally, we underline some related works on bunkbed graphs. In the
random walk field, an analogical problem of the reaching time of a
random walk has been studied, see \cite{bollobas1997random,haggstrom1998conjecture,van2006some}.
In the random directed graph field, it has been shown that it is equivalent
to study random orientation of edges on the graph and percolation
on the graph with Bernouilli paramater $1/2$ (as it is in our case),
see \cite{karp1990transitive,linusson2011percolation,mcdiarmid1981general}. 

A first approach of the problem would be to study the ratio of the
probability of connection, see $\left(\ref{eq:ratioProbaConnection}\right)$,
and to study the derivative according to $p$, the Bernouilli parameter
associated to the probability of opening an edge. Let $u=\left(x,0\right)$
and $v=\left(y,0\right)$ be two vertices on the lower level and define
the vertex $v'$ as $v'=\left(x,1\right)$, the vertex above $v$,
then the ratio of the probability of connection is written as: 
\begin{equation}
\frac{\mathbb{P}_{p}\left(u\leftrightarrow v'\right)}{\mathbb{P}_{p}\left(u\leftrightarrow v\right)}\label{eq:ratioProbaConnection}
\end{equation}
As a result of the previous remark whenever $p$ tends to $0$, the
ratio of $\left(\ref{eq:ratioProbaConnection}\right)$ tends to $0$,
and clearly when $p$ is equal to $1$, this ratio is equal to $1$.
Because the events $\left\{ u\leftrightarrow v\right\} $ and $\left\{ u\leftrightarrow v'\right\} $
are increasing events, the derivatives can be studied using Russo's
formula. However, even using Russo's formula, see \cite{margulis1974probabilistic,russo1981critical},
it reveals itself strenuous to study it. One can notice that the derivative
of $\mathbb{P}_{p}\left(u\leftrightarrow v\right)$ cannot be always
greater than the derivative of $\mathbb{P}_{p}\left(u\leftrightarrow v'\right)$,
since they are both equal to $0$ when $p=0$ and equal to $1$ when
$p=1$. However, we conjecture here that the derivative of the ratio
is increasing, meaning that for all $0<p<1$:
\[
\mathbb{P}_{p}\left[u\leftrightarrow v'\right]\partial_{p}\mathbb{P}_{p}\left[u\leftrightarrow v\right]\leqslant\mathbb{P}_{p}\left[u\leftrightarrow v\right]\partial_{p}\mathbb{P}_{p}\left[u\leftrightarrow v'\right]
\]

From now on, we note $G=\left(V,E\right)$ the bunkbed graph of the
complete graph and we label the vertices of $V$ from 1 to $2n$ such
that $V=\left\{ s_{i},\,i\in\left[1;2n\right]\cap\mathbb{N}\right\} $
and $\forall i,j\in\left[1;n\right]\cap\mathbb{N}$, $s_{i}\sim s_{j}$
and $s_{i}\sim s_{i+n}$ where $\sim$ is the neighbour relation,
$x\sim y\Leftrightarrow\left\{ x,y\right\} \in E$. We can consider
that the vertices labelled from 1 to $n$ are the vertices of the
lower level (or level 1) and the vertices labelled from $n+1$ to
$2n$ are the vertices of the upper level (or level 2). In the complete
graph, two vertices play the same role. Therefore, it is enough to
prove the bunkbed conjecture for two vertices, here $s_{1}$ and $s_{n}$.
Moreover, it is trivial to see that $\mathbb{P}\left(s_{1}\leftrightarrow s_{1}\right)=1\geq\mathbb{P}\left(s_{1}\leftrightarrow s_{n+1}\right)$.

Our approach, based on combinatorics since $p=1/2$, is to decompose
the graph into different appropriate classes which solve the bunbed
conjecture. We count the number of ways to connect $s_{1}$ to $s_{n}$
and to connect $s_{1}$ and $s_{2n}$. The idea of the proof is the
following. We define the \emph{main component} as the connected component
that contains the vertex $s_{1}$ and all the vertices that are connected
to $s_{1}$ by an open path. Then, we distinguish different classes
of main component depending on the number of vertices on the lower
level, on the upper level, and depending on the number of parallel
vertices, notion that will be defined later. Lemma $\ref{lem:LemmeClef1}$
will give the number of ways to connect the vertices of a main component.
Lemmas $\ref{lem:LemmeClef2_1}$ and $\ref{lem:LemmeClef2_2}$ will
give the number of configurations containing a main component with
$x$ vertices on the lower level, $y$ vertices on the upper level
and a key argument to properly add them together. Finally, we prove
the main Theorem in section 4.

\section{Covering Graphs}

\begin{figure}
\definecolor{uququq}{rgb}{0.25098039215686274,0.25098039215686274,0.25098039215686274}   \begin{tikzpicture}[line cap=round,scale=0.5, line join=round,>=triangle 45,x=1.0cm,y=1.0cm] \clip(15.45711045620525,1.0094304872935989) rectangle (29.264907256985282,10.595008431636728); \fill[color=uququq,fill=uququq,fill opacity=0.1] (18.945670332138878,9.68891624294028) -- (17.033618989421115,8.135374526982087) -- (17.978306467516347,5.978306467516348) -- (19.978306467516347,5.978306467516348) -- (20.897556077829943,8.095540124008803) -- cycle; \fill[color=uququq,fill=uququq,fill opacity=0.1] (18.96736386462253,5.710609775423925) -- (17.055312521904767,4.157068059465733) -- (18.,2.) -- (20.,2.) -- (20.919249610313596,4.117233656492446) -- cycle; \fill[color=uququq,fill=uququq,fill opacity=0.1] (25.352206533612073,9.636951268385594) -- (23.44015519089431,8.083409552427405) -- (24.384842668989542,5.926341492961671) -- (27.30409227930314,8.04357514945412) -- cycle; \fill[color=uququq,fill=uququq,fill opacity=0.1] (27.30409227930314,4.043575149454118) -- (24.384842668989542,1.9263414929616722) -- (26.384842668989542,1.9263414929616722) -- cycle; \draw [color=uququq] (4.9673638646225315,5.710609775423926)-- (3.055312521904767,4.157068059465734); \draw [color=uququq] (3.055312521904767,4.157068059465734)-- (4.,2.); \draw [color=uququq] (4.,2.)-- (6.,2.); \draw [color=uququq] (6.,2.)-- (6.919249610313596,4.117233656492448); \draw [color=uququq] (4.9673638646225315,5.710609775423926)-- (6.919249610313596,4.117233656492448); \draw [color=uququq] (4.9673638646225315,5.710609775423926)-- (4.,2.); \draw [color=uququq] (4.,2.)-- (6.919249610313596,4.117233656492448); \draw [color=uququq] (6.919249610313596,4.117233656492448)-- (3.055312521904767,4.157068059465734); \draw [color=uququq] (3.055312521904767,4.157068059465734)-- (6.,2.); \draw [color=uququq] (4.9673638646225315,5.710609775423926)-- (6.,2.); \draw [color=uququq] (4.945670332138879,9.68891624294028)-- (3.0336189894211154,8.135374526982085); \draw [color=uququq] (3.0336189894211154,8.135374526982085)-- (3.978306467516348,5.978306467516349); \draw [color=uququq] (3.978306467516348,5.978306467516349)-- (5.978306467516348,5.978306467516349); \draw [color=uququq] (5.978306467516348,5.978306467516349)-- (6.897556077829944,8.095540124008801); \draw [color=uququq] (4.945670332138879,9.68891624294028)-- (6.897556077829944,8.095540124008801); \draw [color=uququq] (4.945670332138879,9.68891624294028)-- (3.978306467516348,5.978306467516349); \draw [color=uququq] (3.978306467516348,5.978306467516349)-- (6.897556077829944,8.095540124008801); \draw [color=uququq] (6.897556077829944,8.095540124008801)-- (3.0336189894211154,8.135374526982085); \draw [color=uququq] (3.0336189894211154,8.135374526982085)-- (5.978306467516348,5.978306467516349); \draw [color=uququq] (4.945670332138879,9.68891624294028)-- (5.978306467516348,5.978306467516349); \draw [color=uququq] (3.0336189894211154,8.135374526982085)-- (3.055312521904767,4.157068059465734); \draw [color=uququq] (3.978306467516348,5.978306467516349)-- (4.,2.); \draw [color=uququq] (5.978306467516348,5.978306467516349)-- (6.,2.); \draw [color=uququq] (6.897556077829944,8.095540124008801)-- (6.919249610313596,4.117233656492448); \draw [dash pattern=on 4pt off 4pt,color=uququq] (4.945670332138879,9.68891624294028)-- (4.9673638646225315,5.710609775423926); \draw [dotted,color=uququq] (11.352206533612074,5.636951268385596)-- (9.440155190894309,4.0834095524274066); \draw [dotted,color=uququq] (9.440155190894309,4.0834095524274066)-- (10.384842668989542,1.9263414929616722); \draw [color=uququq] (10.384842668989542,1.9263414929616722)-- (12.384842668989542,1.9263414929616722); \draw [color=uququq] (12.384842668989542,1.9263414929616722)-- (13.30409227930314,4.043575149454119); \draw [dotted,color=uququq] (11.352206533612074,5.636951268385596)-- (13.30409227930314,4.043575149454119); \draw [dotted,color=uququq] (11.352206533612074,5.636951268385596)-- (10.384842668989542,1.9263414929616722); \draw [color=uququq] (10.384842668989542,1.9263414929616722)-- (13.30409227930314,4.043575149454119); \draw [dotted,color=uququq] (13.30409227930314,4.043575149454119)-- (9.440155190894309,4.0834095524274066); \draw [dotted,color=uququq] (9.440155190894309,4.0834095524274066)-- (12.384842668989542,1.9263414929616722); \draw [dotted,color=uququq] (11.352206533612074,5.636951268385596)-- (12.384842668989542,1.9263414929616722); \draw [color=uququq] (11.35220653361207,9.636951268385594)-- (9.440155190894309,8.083409552427405); \draw [color=uququq] (9.440155190894309,8.083409552427405)-- (10.384842668989542,5.926341492961672); \draw [dotted,color=uququq] (10.384842668989542,5.926341492961672)-- (12.384842668989542,5.926341492961672); \draw [dotted,color=uququq] (12.384842668989542,5.926341492961672)-- (13.30409227930314,8.04357514945412); \draw [color=uququq] (11.35220653361207,9.636951268385594)-- (13.30409227930314,8.04357514945412); \draw [color=uququq] (11.35220653361207,9.636951268385594)-- (10.384842668989542,5.926341492961672); \draw [color=uququq] (10.384842668989542,5.926341492961672)-- (13.30409227930314,8.04357514945412); \draw [color=uququq] (13.30409227930314,8.04357514945412)-- (9.440155190894309,8.083409552427405); \draw [dotted,color=uququq] (9.440155190894309,8.083409552427405)-- (12.384842668989542,5.926341492961672); \draw [dotted,color=uququq] (11.35220653361207,9.636951268385594)-- (12.384842668989542,5.926341492961672); \draw [dotted,color=uququq] (9.440155190894309,8.083409552427405)-- (9.440155190894309,4.0834095524274066); \draw [color=uququq] (10.384842668989542,5.926341492961672)-- (10.384842668989542,1.9263414929616722); \draw [dotted,color=uququq] (12.384842668989542,5.926341492961672)-- (12.384842668989542,1.9263414929616722); \draw [color=uququq] (13.30409227930314,8.04357514945412)-- (13.30409227930314,4.043575149454119); \draw [dotted,color=uququq] (11.35220653361207,9.636951268385594)-- (11.352206533612074,5.636951268385596); \draw [color=uququq] (18.96736386462253,5.710609775423925)-- (17.055312521904767,4.157068059465733); \draw [color=uququq] (17.055312521904767,4.157068059465733)-- (18.,2.); \draw [color=uququq] (18.,2.)-- (20.,2.); \draw [color=uququq] (20.,2.)-- (20.919249610313596,4.117233656492446); \draw [color=uququq] (18.96736386462253,5.710609775423925)-- (20.919249610313596,4.117233656492446); \draw [color=uququq] (18.96736386462253,5.710609775423925)-- (18.,2.); \draw [color=uququq] (18.,2.)-- (20.919249610313596,4.117233656492446); \draw [color=uququq] (20.919249610313596,4.117233656492446)-- (17.055312521904767,4.157068059465733); \draw [color=uququq] (17.055312521904767,4.157068059465733)-- (20.,2.); \draw [color=uququq] (18.96736386462253,5.710609775423925)-- (20.,2.); \draw [color=uququq] (18.945670332138878,9.68891624294028)-- (17.033618989421115,8.135374526982087); \draw [color=uququq] (17.033618989421115,8.135374526982087)-- (17.978306467516347,5.978306467516348); \draw [color=uququq] (17.978306467516347,5.978306467516348)-- (19.978306467516347,5.978306467516348); \draw [color=uququq] (19.978306467516347,5.978306467516348)-- (20.897556077829943,8.095540124008803); \draw [color=uququq] (18.945670332138878,9.68891624294028)-- (20.897556077829943,8.095540124008803); \draw [color=uququq] (18.945670332138878,9.68891624294028)-- (17.978306467516347,5.978306467516348); \draw [color=uququq] (17.978306467516347,5.978306467516348)-- (20.897556077829943,8.095540124008803); \draw [color=uququq] (20.897556077829943,8.095540124008803)-- (17.033618989421115,8.135374526982087); \draw [color=uququq] (17.033618989421115,8.135374526982087)-- (19.978306467516347,5.978306467516348); \draw [color=uququq] (18.945670332138878,9.68891624294028)-- (19.978306467516347,5.978306467516348); \draw [color=uququq] (17.033618989421115,8.135374526982087)-- (17.055312521904767,4.157068059465733); \draw [color=uququq] (17.978306467516347,5.978306467516348)-- (18.,2.); \draw [color=uququq] (19.978306467516347,5.978306467516348)-- (20.,2.); \draw [color=uququq] (20.897556077829943,8.095540124008803)-- (20.919249610313596,4.117233656492446); \draw [dash pattern=on 4pt off 4pt,color=uququq] (18.945670332138878,9.68891624294028)-- (18.96736386462253,5.710609775423925); \draw [dotted,color=uququq] (25.352206533612073,5.636951268385595)-- (23.44015519089431,4.083409552427406); \draw [dotted,color=uququq] (23.44015519089431,4.083409552427406)-- (24.384842668989542,1.9263414929616722); \draw [color=uququq] (24.384842668989542,1.9263414929616722)-- (26.384842668989542,1.9263414929616722); \draw [color=uququq] (26.384842668989542,1.9263414929616722)-- (27.30409227930314,4.043575149454118); \draw [dotted,color=uququq] (25.352206533612073,5.636951268385595)-- (27.30409227930314,4.043575149454118); \draw [dotted,color=uququq] (25.352206533612073,5.636951268385595)-- (24.384842668989542,1.9263414929616722); \draw [color=uququq] (24.384842668989542,1.9263414929616722)-- (27.30409227930314,4.043575149454118); \draw [dotted,color=uququq] (27.30409227930314,4.043575149454118)-- (23.44015519089431,4.083409552427406); \draw [dotted,color=uququq] (23.44015519089431,4.083409552427406)-- (26.384842668989542,1.9263414929616722); \draw [dotted,color=uququq] (25.352206533612073,5.636951268385595)-- (26.384842668989542,1.9263414929616722); \draw [color=uququq] (25.352206533612073,9.636951268385594)-- (23.44015519089431,8.083409552427405); \draw [color=uququq] (23.44015519089431,8.083409552427405)-- (24.384842668989542,5.926341492961671); \draw [dotted,color=uququq] (24.384842668989542,5.926341492961671)-- (26.384842668989542,5.926341492961671); \draw [dotted,color=uququq] (26.384842668989542,5.926341492961671)-- (27.30409227930314,8.04357514945412); \draw [color=uququq] (25.352206533612073,9.636951268385594)-- (27.30409227930314,8.04357514945412); \draw [color=uququq] (25.352206533612073,9.636951268385594)-- (24.384842668989542,5.926341492961671); \draw [color=uququq] (24.384842668989542,5.926341492961671)-- (27.30409227930314,8.04357514945412); \draw [color=uququq] (27.30409227930314,8.04357514945412)-- (23.44015519089431,8.083409552427405); \draw [dotted,color=uququq] (23.44015519089431,8.083409552427405)-- (26.384842668989542,5.926341492961671); \draw [dotted,color=uququq] (25.352206533612073,9.636951268385594)-- (26.384842668989542,5.926341492961671); \draw [dotted,color=uququq] (23.44015519089431,8.083409552427405)-- (23.44015519089431,4.083409552427406); \draw [color=uququq] (24.384842668989542,5.926341492961671)-- (24.384842668989542,1.9263414929616722); \draw [dotted,color=uququq] (26.384842668989542,5.926341492961671)-- (26.384842668989542,1.9263414929616722); \draw [color=uququq] (27.30409227930314,8.04357514945412)-- (27.30409227930314,4.043575149454118); \draw [dotted,color=uququq] (25.352206533612073,9.636951268385594)-- (25.352206533612073,5.636951268385595); \draw [color=uququq] (18.945670332138878,9.68891624294028)-- (17.033618989421115,8.135374526982087); \draw [color=uququq] (17.033618989421115,8.135374526982087)-- (17.978306467516347,5.978306467516348); \draw [color=uququq] (17.978306467516347,5.978306467516348)-- (19.978306467516347,5.978306467516348); \draw [color=uququq] (19.978306467516347,5.978306467516348)-- (20.897556077829943,8.095540124008803); \draw [color=uququq] (20.897556077829943,8.095540124008803)-- (18.945670332138878,9.68891624294028); \draw [color=uququq] (18.96736386462253,5.710609775423925)-- (17.055312521904767,4.157068059465733); \draw [color=uququq] (17.055312521904767,4.157068059465733)-- (18.,2.); \draw [color=uququq] (18.,2.)-- (20.,2.); \draw [color=uququq] (20.,2.)-- (20.919249610313596,4.117233656492446); \draw [color=uququq] (20.919249610313596,4.117233656492446)-- (18.96736386462253,5.710609775423925); \draw [color=uququq] (25.352206533612073,9.636951268385594)-- (23.44015519089431,8.083409552427405); \draw [color=uququq] (23.44015519089431,8.083409552427405)-- (24.384842668989542,5.926341492961671); \draw [color=uququq] (24.384842668989542,5.926341492961671)-- (27.30409227930314,8.04357514945412); \draw [color=uququq] (27.30409227930314,8.04357514945412)-- (25.352206533612073,9.636951268385594); \draw [color=uququq] (27.30409227930314,4.043575149454118)-- (24.384842668989542,1.9263414929616722); \draw [color=uququq] (24.384842668989542,1.9263414929616722)-- (26.384842668989542,1.9263414929616722); \draw [color=uququq] (26.384842668989542,1.9263414929616722)-- (27.30409227930314,4.043575149454118); \end{tikzpicture}\caption{The bunkbed graph of $K_{5}$ and an element of $G_{3,4,2}$}
\label{FigureQuiDechire}
\end{figure}
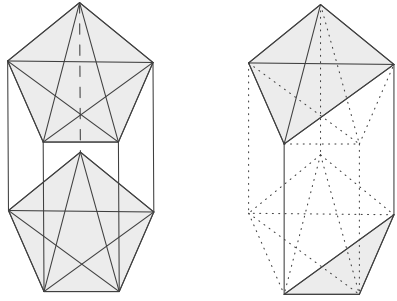
 We recall that a covering graph of a graph $G=\left(V,E\right)$
is a subgraph $G'=\left(V',E'\right)$ such that $G'$ is connected,
$V'=V$ and $E'\subseteq E$. From now on, $G$ will be used to refer
to the bunkbed graph of the complete graph. We classify the subgraphs
of $G$ according to the nomber of vertices in the upper and lower
level, as well as the number of parallel vertices. We define $G_{x,y,z}$
the set of connected subgraph of $G$ such that $\forall G'=\left(V',E'\right)\in G_{x,y,z}$
:
\begin{enumerate}
\item $\exists\lambda_{1},...,\lambda_{x}\in\left\{ 1,...,n\right\} ,\exists\mu_{1},...,\mu_{y}\in\left\{ n+1,...,2n\right\} $
such that $\forall i\in\left[1;x\right]$ $s_{\lambda_{i}}\in V'$
and $\forall j\in\left[1;y\right]$, $s_{\mu_{j}}\in V'$, $\cup_{i}\{s_{\lambda_{i}}\}\cup_{j}\{s_{\mu_{j}}\}=V$
and $\#V=x+y$
\item $\exists\lambda_{1},...,\lambda_{z}\in\left\{ 1,...,n\right\} $ such
that $\forall i,j\leqslant z,\,\lambda_{i}\neq\lambda_{j}$ and $\left\{ s_{\lambda_{i}},s_{\lambda_{i}+n}\right\} \subset V'$
\item $\forall x,y\in V'$, $\left\{ x,y\right\} \in E'$ iff $\left\{ x,y\right\} \in E$
\end{enumerate}
Graphs of $G_{x,y,z}$ can be seen as extraction of subgraph of the
bunkbed graph $G.$ Condition 1 insures that there are exactly $x$
vertices on the lower level and exactly $y$ vertices on the upper
level. Condition number 2 insures that exactly $z$ vertices among
the $y$ vertices on the upper level are above the $x$ vertices of
the lower level. Finally, condition 3 insures that vertices present
in the extraction comes along with the corresponding edges. We will
say that a graph of $G_{x,y,z}$ has $x$ vertices of level 1 and
$y$ vertices of level 2 and $z$ parallel vertices (e.g. figure $\ref{FigureQuiDechire}$).
Moreover, one can see that two graphs $G_{1}$, $G_{2}$ of $G_{x,y,z}$
are isomorphs. 

We define the function $GC:\mathbb{N}^{3}\mapsto\mathbb{N}$ which
gives the number $GC\left(x,y,z\right)$ of covering graphs of a graph
in $G_{x,y,z}$. 
\begin{lem}
\label{lem:LemmeClef1}Fix $x,x',y,y',z\in\left[0;n\right]\cap\mathbb{N}$
such that $x+y=x'+y'$ and $z\leqslant\min\left(x,x',y,y'\right)$.
If $\left|x-y\right|\geqslant\left|x'-y'\right|$ then $GC\left(x,y,z\right)\geqslant GC\left(x',y',z\right)$\end{lem}
\begin{rem}
The function $GC$ is symmetric in its first two coordinates, meaning
that for all $x,y,z$, we have $GC\left(x,y,z\right)=GC\left(y,x,z\right)$.\end{rem}
\begin{proof}
To prove this lemma, by iteration, it is sufficient to prove the inequality
$GC\left(x+1,y,z\right)\geqslant GC\left(x,y+1,z\right)$ for $x>y\geqslant z$.
For this matter, we need to give an upper bound and a lower bound
of $GC\left(x,y,z\right)$. As an upper bound we use a trivial one:
we bound the number of covering graphs by the number of possible graph
knowing that one vertical edge need to be open. Since in a complete
graph with $n$ vertices, there are $n\left(n-1\right)/2$ edges,
therefore we have:

\[
GC\left(x,y,z\right)\leqslant2^{\frac{x\left(x-1\right)}{2}}\left(2^{z}-1\right)2^{\frac{y\left(y-1\right)}{2}}
\]
As a lower bound, we consider the only case where we connect by at
least one vertical edge a covering graph $K_{x}$ and a covering graph
of $K_{y}$:
\[
GC\left(x,y,z\right)\geqslant GC\left(x,0,0\right)\times\left(2^{z}-1\right)\times GC\left(0,y,0\right)
\]
Therefore, one has:
\begin{eqnarray*}
\frac{GC\left(x+1,y,z\right)}{GC\left(x,y+1,z\right)} & \geqslant & \frac{GC\left(x+1,0,0\right)\times\left(2^{z}-1\right)\times GC\left(0,y,0\right)}{2^{\frac{x\left(x-1\right)}{2}}\left(2^{z}-1\right)2^{\frac{y\left(y+1\right)}{2}}}\\
 & = & GC\left(x+1,0,0\right)\times2^{-\frac{x\left(x-1\right)}{2}}\times GC\left(0,y,0\right)\times2^{-\frac{y\left(y-1\right)}{2}}\times2^{-y}
\end{eqnarray*}
By \cite{flajolet2009analytic}, the number of covering graphs or
number of connected labelled graph with $n$ vertices is given by
the following approximation: 
\[
GC\left(n,0,0\right)=2^{\frac{n\left(n-1\right)}{2}}\left(1-2n2^{-n}+o\left(2^{-n}\right)\right)
\]
We slightly modify this approximation in the case $n\geqslant7$:
\[
GC\left(n,0,0\right)\geqslant2^{\frac{n\left(n-1\right)}{2}}\left(1-3n2^{-n}\right)
\]
Then, we split the problems into different cases. First case, whenever
$x\geqslant10$ and $y\geqslant7$. Since $x\geqslant y+1$, one has:
\begin{eqnarray*}
\frac{GC\left(x+1,y,z\right)}{GC\left(x,y+1,z\right)} & \geqslant & 2^{x-y}\left(1-3x2^{-x}\right)\left(1-3y2^{-y}\right)\\
 & \geqslant & 2\times\left(1-3\times10\times2^{-10}\right)\times\left(1-3\times7\times2^{-7}\right)\geqslant1
\end{eqnarray*}
Second case, when $x\geqslant10$ and $y<7$, then $x\geqslant y+4$:
\begin{eqnarray*}
\frac{GC\left(x+1,y,z\right)}{GC\left(x,y+1,z\right)} & \geqslant & 2^{x-y-1}\left(1-3x2^{-x}\right)\\
 & \geqslant & 2^{3}\times\left(1-3\times10\times2^{-10}\right)\geqslant1
\end{eqnarray*}
Finally, when $10>x>y$, we have computed the result by computer which
ends the proof.
\end{proof}

\section{Result on the Size of the Classes}

We define $G^{1}$the set of connected subgraphs of $G$ containing
the vertices $s_{1}$ and $s_{n}$, and $G^{2}$ the set of connected
subgraphs of $G$ containing he vertices $s_{1}$ and $s_{2n}$. Moreover,
we define the set of graph $G_{x,y,z}^{1}=G^{1}\cap G_{x,y,z}$ and
the set $G_{x,y,z}^{2}=G^{2}\cap G_{x,y,z}$. We define the functions
$q_{1}:\mathbb{N}^{3}\mapsto\mathbb{N}$ and $q_{2}:\mathbb{N}^{3}\mapsto\mathbb{N}$
such that $q_{1}\left(x,y,z\right)=\#G_{x,y,z}^{1}$ and $q_{2}\left(x,y,z\right)=\#G_{x,y,z}^{2}$.
Finally, we define the function $q:\mathbb{N}^{3}\mapsto\mathbb{N}$
such that 
\begin{equation}
q\left(x,y,z\right)=\begin{cases}
q_{1}\left(x,y,z\right)-q_{2}\left(x,y,z\right)+q_{1}\left(y,x,z\right)-q_{2}\left(y,x,z\right) & \mbox{if }x\neq y\\
q_{1}\left(x,x,z\right)-q_{2}\left(x,x,z\right) & \mbox{ if }x=y
\end{cases}\label{eq:definitiondeQ}
\end{equation}
Before giving the main result on the function $q$, we give two preliminary
results on the exact value of the functions $q_{1}$and $q_{2}$.
\begin{lem}
\label{lem:lemmeQ1}For all $x,y\geqslant z\geqslant1$ such that
$x+y-z\leqslant n$
\[
q_{1}(x,y,z)=\frac{\left(n-2\right)!x\left(x-1\right)}{\left(x-z\right)!z!\left(n-x-y+z\right)!\left(y-z\right)!}
\]
\end{lem}
\begin{proof}
First, if $z=0$ and $y>0$, then a graph cannot be connected. Moreover,
$s_{0}$ and $s_{n}$ have to be in the set of vertices of the graph
of $G_{x,y,z}^{1}$, $x$ has to be greater that 2 otherwise $G_{x,y,z}^{1}$
is an empty set. Then, we have to choose the $x-2$ vertices of level
1 among the the $n-2$ vertices left, distribute $z$ vertices of
level 2 on top of the $x$ vertices previously chosen, and choose
$y-z$ vertices among the $n-x$ vertices left. Therefore, we can
write:
\[
q_{1}(x,y,z)=\binom{n-2}{x-2}\times\binom{x}{z}\times\binom{n-x}{y-z}\times\mathds{1}_{x\geqslant2}
\]
Then, one can note that for all $k\in\mathbb{N}$, the following equality
holds:
\begin{equation}
\frac{1}{\left(x-k\right)!}\mathds{1}_{x\geqslant k}=\frac{1}{x!}\prod_{i=0}^{k}\left(x-i\right)\label{eq:factorielleetindicatrice}
\end{equation}
So we can write:
\begin{eqnarray*}
q_{1}(x,y,z) & = & \frac{\left(n-2\right)!}{\left(n-x\right)!\left(x-2\right)!}\times\frac{x!}{\left(x-z\right)!z!}\times\frac{\left(n-x\right)!}{\left(n-x-y+z\right)!\left(y-z\right)!}\mathds{1}_{x\geqslant2}\\
 & = & \frac{\left(n-2\right)!x\left(x-1\right)}{\left(x-z\right)!z!\left(n-x-y+z\right)!\left(y-z\right)!}
\end{eqnarray*}
\end{proof}
\begin{lem}
\label{lem:lemmeQ2}For all $x,y\geqslant z\geqslant1$ such that
$x+y-z\leqslant n$:
\[
q_{2}\left(x,y,z\right)=\frac{\left(n-2\right)!\left(xy-z\right)}{\left(x-z\right)!z!\left(n-x-y+z\right)!\left(y-z\right)!}
\]
\end{lem}
\begin{proof}
First, we notice that for any graph in $G_{x,y,z}^{2}$, $s_{0}$
and $s_{2n}$ belong to the set of vertices. Then, to count the number
of graphs $G$ in $G_{x,y,z}^{2}$, we distinguish 4 different cases:
either $s_{n}$ and $s_{n+1}$ belong to $G$; either $s_{n}$ belongs
to $G$ but not $s_{n+1}$; either $s_{n}$ does not belong to $G$
but $s_{n+1}$ does; either $s_{n}$ and $s_{n+1}$ don't belong to
$G$. We can write:

\begin{eqnarray*}
q_{2}\left(x,y,z\right) & = & \binom{n-2}{x-2}\times\binom{x-2}{z-2}\times\binom{n-x}{y-z}\times\mathds{1}_{x\geqslant2,y\geqslant2,z\geqslant2}\\
 &  & +\binom{n-2}{x-2}\times\binom{x-2}{z-1}\times\binom{n-x}{y-z}\times\mathds{1}_{x>\max\left(1,z\right),y<n}\\
 &  & +\binom{n-2}{x-1}\times\binom{x-1}{z-1}\times\binom{n-x-1}{y-z-1}\times\mathds{1}_{x<n,y>\max\left(1,z\right)}\\
 &  & +\binom{n-2}{x-1}\times\binom{x-1}{z}\times\binom{n-x-1}{y-z-1}\times\mathds{1}_{\max\left(1,z\right)<x<n,\max\left(1,z\right)<y<n}
\end{eqnarray*}
Therefore, using $\left(\ref{eq:factorielleetindicatrice}\right)$,
we have:
\begin{eqnarray*}
q_{2}\left(x,y,z\right) & = & \frac{\left(n-2\right)!}{\left(x-z\right)!z!\left(n-x-y+z\right)!\left(y-z\right)!}\times z\left(z-1\right)\\
 &  & +\frac{\left(n-2\right)!}{\left(x-z\right)!z!\left(n-x-y+z\right)!\left(y-z\right)!}\times z\left(x-z\right)\\
 &  & +\frac{\left(n-2\right)!}{\left(x-z\right)!z!\left(n-x-y+z\right)!\left(y-z\right)!}\times z\left(y-z\right)\\
 &  & +\frac{\left(n-2\right)!}{\left(x-z\right)!z!\left(n-x-y+z\right)!\left(y-z\right)!}\times\left(x-z\right)\left(y-z\right)\\
 & = & \frac{\left(n-2\right)!}{\left(x-z\right)!z!\left(n-x-y+z\right)!\left(y-z\right)!}\left(xy-z\right)
\end{eqnarray*}
Using lemmas $\ref{lem:lemmeQ1}$ and $\ref{lem:lemmeQ2}$, we have
that:
\[
q_{1}\left(x,y,z\right)-q_{2}\left(x,y,z\right)=\frac{\left(n-2\right)!\left(x^{2}-x-xy+z\right)}{\left(x-z\right)!z!\left(n-x-y+z\right)!\left(y-z\right)!}
\]
Recall the definition of the function $q$ in $\left(\ref{eq:definitiondeQ}\right)$,
we have for all $x\geqslant z\geqslant1$:
\[
q\left(x,x,z\right)=\frac{\left(n-2\right)!\left(z-x\right)}{\left(x-z\right)!\left(x-z\right)!z!\left(n-2x+z\right)!}
\]
And for all $x,y\geqslant z\geqslant1$:
\begin{equation}
q\left(x,y,z\right)=\frac{\left(n-2\right)!\left(x^{2}-2xy+y^{2}-x-y+2z\right)}{\left(x-z\right)!\left(y-z\right)!z!\left(n-x-y+z\right)!}\label{eq:ValeurDeQ}
\end{equation}
Note that:
\begin{equation}
q\left(x,y,z\right)\leqslant0\Leftrightarrow x\in\left[y+\frac{1-\sqrt{8y-8z+1}}{2};y+\frac{1+\sqrt{8y-8z+1}}{2}\right]\label{eq:qnegatif}
\end{equation}
\end{proof}
\begin{lem}
\label{lem:LemmeClef2_1}For all $k\geqslant z$, the following inequality
holds: 
\[
\sum_{i=0}^{k-z}q\left(k+i,k-i,z\right)=0
\]
\end{lem}
\begin{proof}
To prove the theorem, it is actually easier to prove that: 
\[
\sum_{i=1}^{k-z}q\left(k+i,k-i,z\right)=-q\left(k,k,z\right)
\]
Using as arguments of $q$ the triplet $\left(k+i,k-i,z\right)$,
some factors of $\left(\ref{eq:ValeurDeQ}\right)$ become independent
of $i$. Indeed, we get the following equality:
\[
q\left(k+i,k-i,z\right)=\frac{4i^{2}-2k+2z}{\left(k+i-z\right)!\left(k-i-z\right)!}\times\frac{\left(n-2\right)!}{z!\left(n-2k+z\right)!}
\]
Therefore, proving the Lemma is equivalent to prove that:
\[
\sum_{i=1}^{k-z}\frac{4i^{2}-2k+2z}{\left(k+i-z\right)!\left(k-i-z\right)!}=\frac{k-z}{\left(k-z\right)!\left(k-z\right)!}
\]
Then, it is enough to see that:
\begin{eqnarray*}
\frac{k-z}{\left(k-z\right)!\left(k-z\right)!} & = & \frac{4-2k+2z}{\left(k+1-z\right)!\left(k-1-z\right)!}\\
 &  & +\frac{3\left(k+2-z\right)}{\left(k+2-z\right)!\left(k-2-z\right)!}\mathds{1}_{k-z\geqslant2}
\end{eqnarray*}
and that for all $k-z>i\geqslant2$:
\begin{eqnarray*}
\frac{\left(2i-1\right)\left(k+i-z\right)}{\left(k+i-z\right)!\left(k-i-z\right)!} & = & \frac{4i^{2}-2k+2z}{\left(k+i-z\right)!\left(k-i-z\right)!}\\
 &  & +\frac{\left(2i+1\right)\left(k+i+1-z\right)}{\left(k+i+1-z\right)!\left(k-i-1-z\right)!}
\end{eqnarray*}
Whenever $i=k-z$, then:
\[
\frac{\left(2i-1\right)\left(k+i-z\right)}{\left(k+i-z\right)!\left(k-i-z\right)!}=\frac{4\left(k-z\right)^{2}-2k+2z}{\left(2k-2z\right)!}
\]
which ends the proof.\end{proof}
\begin{lem}
\label{lem:LemmeClef2_2}For all $k\geqslant z$, the following inequality
holds: 
\[
\sum_{i=0}^{k-z}q\left(k+i+1,k-i,z\right)=0
\]
\end{lem}
\begin{proof}
The proof goes in the same way as lemma $\ref{lem:LemmeClef2_1}$.
Indeed, it is enough to prove that:
\[
\sum_{i=1}^{k-z}\frac{2i^{2}+2i-k+z}{\left(k+i+1-z\right)!\left(k-i-z\right)!}=\frac{k-z}{\left(k+1-z\right)!\left(k-z\right)!}
\]
Then, we have:
\begin{eqnarray*}
\frac{k-z}{\left(k+1-z\right)!\left(k-z\right)!} & = & \frac{4-k+z}{\left(k+2-z\right)!\left(k-1-z\right)!}\\
 &  & +\frac{2\left(k-z+3\right)}{\left(k+3-z\right)!\left(k-2-z\right)!}\mathds{1}_{k-z\geqslant2}
\end{eqnarray*}
Then, for all $k-z>i\geqslant2$, the following equality holds:
\begin{eqnarray*}
\frac{i\left(k+i+1-z\right)}{\left(k+i-z\right)!\left(k-i-z\right)!} & = & \frac{2i^{2}+2i-k-z}{\left(k+i-z\right)!\left(k-i-z\right)!}\\
 &  & +\frac{\left(i+1\right)\left(k+i+2-z\right)}{\left(k+i+1-z\right)!\left(k-i-1-z\right)!}
\end{eqnarray*}
Whenever $i=k-z$,
\[
\frac{i\left(k+i+1-z\right)}{\left(k+i+1-z\right)!\left(k-i-1-z\right)!}=\frac{2\left(k-z\right)^{2}+2\left(k-z\right)-k+z}{\left(2k-2z+1\right)!}
\]
which ends the proof.
\end{proof}

\section{Proof of the main theorem}

\begin{figure}
\definecolor{ttqqqq}{rgb}{0.2,0.,0.} \definecolor{ttttff}{rgb}{0.2,0.2,1.} \definecolor{ffqqtt}{rgb}{1.,0.,0.2} \definecolor{qqqqff}{rgb}{0.,0.,1.} \definecolor{qqccqq}{rgb}{0.,0.8,0.} \begin{tikzpicture}[line cap=round,line join=round,>=triangle 45,x=1.0cm,y=1.0cm] \clip(0.36537986288162655,0.35036057006642657) rectangle (9.863192467056201,5.603471738795075); \draw [dotted,color=ffqqtt] (5.98673615483871,3.6716805519544597)-- (5.993713257948041,1.7105411405408115); \draw [color=qqccqq] (7.182757001290323,2.8977847101328282)-- (7.189734104399655,0.9366452987191802); \draw [dotted,color=ffqqtt] (8.56638896091082,3.64822916280835)-- (8.573366064020151,1.6870897513947014); \draw [dotted,color=ffqqtt] (7.323465336166985,4.398673615483872)-- (7.3304424392763154,2.4375342040702224); \draw [dotted,color=ttttff] (7.323465336166985,4.398673615483872)-- (5.98673615483871,3.6716805519544597); \draw [dash pattern=on 1pt off 1pt,color=ffqqtt] (7.182757001290323,2.8977847101328282)-- (5.98673615483871,3.6716805519544597); \draw [color=ttttff] (7.323465336166985,4.398673615483872)-- (8.56638896091082,3.64822916280835); \draw [dotted,color=ffqqtt] (7.182757001290323,2.8977847101328282)-- (8.56638896091082,3.64822916280835); \draw [color=qqccqq] (8.573366064020151,1.6870897513947014)-- (7.189734104399655,0.9366452987191802); \draw [color=qqccqq] (7.189734104399655,0.9366452987191802)-- (5.993713257948041,1.7105411405408115); \draw [dotted,color=qqccqq] (5.993713257948041,1.7105411405408115)-- (7.3304424392763154,2.4375342040702224); \draw [color=qqccqq] (7.3304424392763154,2.4375342040702224)-- (8.573366064020151,1.6870897513947014); \draw [dotted,color=ttqqqq] (1.273006936470588,3.7185833302466773)-- (1.279984039579919,1.7574439188330289); \draw [color=ttqqqq] (2.4690277829222005,2.944687488425046)-- (2.4760048860315322,0.9835480770113976); \draw [dotted,color=ttqqqq] (3.852659742542702,3.6951319411005685)-- (3.859636845652034,1.7339925296869187); \draw [dotted,color=ttqqqq] (2.609736117798862,4.4455763937760935)-- (2.616713220908192,2.4844369823624395); \draw [dotted,color=ttqqqq] (2.609736117798862,4.4455763937760935)-- (1.273006936470588,3.7185833302466773); \draw [dash pattern=on 1pt off 1pt,color=ttqqqq] (2.4690277829222005,2.944687488425046)-- (1.273006936470588,3.7185833302466773); \draw [color=ttqqqq] (2.609736117798862,4.4455763937760935)-- (3.852659742542702,3.6951319411005685); \draw [dotted,color=ttqqqq] (2.4690277829222005,2.944687488425046)-- (3.852659742542702,3.6951319411005685); \draw [color=ttqqqq] (3.859636845652034,1.7339925296869187)-- (2.4760048860315322,0.9835480770113976); \draw [color=ttqqqq] (2.4760048860315322,0.9835480770113976)-- (1.279984039579919,1.7574439188330289); \draw [dotted,color=ttqqqq] (1.279984039579919,1.7574439188330289)-- (2.616713220908192,2.4844369823624395); \draw [color=ttqqqq] (2.616713220908192,2.4844369823624395)-- (3.859636845652034,1.7339925296869187); \begin{scriptsize} \draw [fill=qqccqq] (5.993713257948041,1.7105411405408115) circle (2.5pt); \draw[color=qqccqq] (6.157872981970811,2.1326661451707896) node {$U$}; \draw [fill=qqccqq] (7.189734104399655,0.9366452987191802) circle (2.5pt); \draw [fill=qqccqq] (8.573366064020151,1.6870897513947014) circle (2.5pt); \draw[color=qqccqq] (8.737525788042918,2.1092147560246794) node {$V$}; \draw [fill=qqccqq] (7.3304424392763154,2.4375342040702224) circle (2.5pt); \draw [fill=qqqqff] (5.98673615483871,3.6716805519544597) circle (2.5pt); \draw [fill=qqccqq] (7.182757001290323,2.8977847101328282) circle (2.5pt); \draw [fill=qqqqff] (8.56638896091082,3.64822916280835) circle (2.5pt); \draw[color=qqqqff] (8.807879955481248,4.079131444297922) node {$V'$}; \draw [fill=qqqqff] (7.323465336166985,4.398673615483872) circle (2.5pt); \draw [fill=ttqqqq] (1.279984039579919,1.7574439188330289) circle (2.5pt); \draw[color=ttqqqq] (1.4910465418949095,2.2499230909013397) node {$U$}; \draw [fill=ttqqqq] (2.4760048860315322,0.9835480770113976) circle (2.5pt); \draw [fill=ttqqqq] (3.859636845652034,1.7339925296869187) circle (2.5pt); \draw[color=ttqqqq] (4.0706993479670155,2.2264717017552296) node {$V$}; \draw [fill=ttqqqq] (2.616713220908192,2.4844369823624395) circle (2.5pt); \draw [fill=ttqqqq] (1.273006936470588,3.7185833302466773) circle (2.5pt); \draw [fill=ttqqqq] (2.4690277829222005,2.944687488425046) circle (2.5pt); \draw [fill=ttqqqq] (3.852659742542702,3.6951319411005685) circle (2.5pt); \draw[color=ttqqqq] (4.141053515405346,4.196388390028472) node {$V'$}; \draw [fill=ttqqqq] (2.609736117798862,4.4455763937760935) circle (2.5pt); \end{scriptsize} \end{tikzpicture}

\caption{Decomposition of a configuration}

\label{FigureDeDecomposition}
\end{figure}
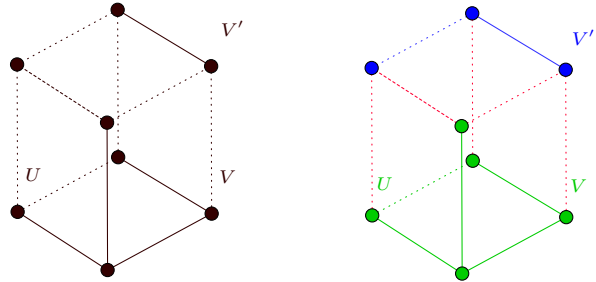
 Recall that $G=\left(V,E\right)$ is the bunkbed graph associated
with the complete graph $K_{n}$. We recall here the idea of the proof
of the main theorem. We split the configuration depending on the number
of vertices of the main component and then split again depending on
the number of vertices on the lower/upper level. Figure $\ref{FigureDeDecomposition}$
gives a decomposition of a configuration and we give some explanation
about the figure. For simplicity we have not drawn a complete graph.
Then, edges drawn in solid lines are open edges and edges drawn in
dotted lines are closed edges. Green vertices and green edges correspond
to the main component. Red edges are the adjacent edges of the main
component that need to be closed, if not, the main component would
expand. Blue vertices and blue edges correspond to the ``outside''
of the main component. In this sense, we define $O\left(x,y,z\right)$
the number of ``outside'' edges when the main component has $x$
vertices in the lower level, $y$ vertices in the upper level and
$z$ parallel vertices. A simple computation gives that:
\[
O\left(x,y,z\right)=n\left(n-x-y\right)+\frac{1}{2}\left(x^{2}-x+y^{2}-y\right)+z
\]
As well as the following relations:
\begin{eqnarray*}
O\left(x,y,z\right) & = & O\left(y,x,z\right)\\
O\left(x+1,y,z\right)-O\left(x,y+1,z\right) & = & x-y
\end{eqnarray*}
The second relation can be understood in the spirit of Lemma $\ref{lem:LemmeClef1}$
through the following statement: $x+y=x'+y'$ and $\left|x-y\right|\geqslant\left|x'-y'\right|$
implies $O\left(x,y,z\right)\geqslant O\left(x',y',z\right)$. All
the quantities developped before allows us to express the probability
of connection between two vertices. First, note that since the Bernouilli
parameter of the percolation $p$ is equal to $1/2$, therefore every
configuration has the same probability, i.e. for any configuration
$\omega$, $\mathbb{P}_{1/2}\left(\omega\right)=2^{-\#E}$. Moreover,
knowing that the main component of $s_{0}$, noted $MC\left(s_{0}\right)$,
is a subgraph of a graph in $G_{x,y,z}^{1}$, there are $GC(x,y,z)$
way for it to be connected, and $O\left(x,y,z\right)$ outside edges
which won't affect the connectivity of $s_{0}$ and $s_{n}$. Furthermore,
there are $q_{1}\left(x,y,z\right)$ to choose the subgraph corresponding
to the main component. Thus, we have the following equalities:
\begin{eqnarray*}
 &  & \mathbb{P}_{1/2}\left(s_{0}\leftrightarrow s_{n}\right)\\
 & = & \sum_{x,y,z}\sum_{MC\left(s_{0}\right)\mbox{ subgraph of }\tilde{G}\in G_{x,y,e}^{1}}2^{O\left(x,y,z\right)}\mbox{\ensuremath{\mathbb{P}\left(\omega\right)}}\\
 & = & 2^{-\#E}\sum_{x,y,z}2^{O\left(x,y,z\right)}GC\left(x,y,z\right)q_{1}\left(x,y,z\right)\\
 & = & 2^{-\#E}\sum_{k\geqslant0}\sum_{z\geqslant0}\sum_{i\in\mathbb{Z}}2^{O\left(k+i,k-i,z\right)}GC\left(k+i,k-i,z\right)q_{1}\left(k+i,k-i,z\right)\\
 &  & \quad+2^{-\#E}\sum_{k\geqslant0}\sum_{z\geqslant0}\sum_{i\in\mathbb{Z}}2^{O\left(k+i+1,k-i,z\right)}GC\left(k+i+1,k-i,z\right)q_{1}\left(k+i+1,k-i,z\right)
\end{eqnarray*}
The third equality is obtained by an operation of renumbering. In
the same way, we have that:
\begin{eqnarray*}
 &  & \mathbb{P}_{1/2}\left(s_{0}\leftrightarrow s_{2n}\right)\\
 & = & 2^{-\#E}\sum_{k\geqslant0}\sum_{z\geqslant0}\sum_{i\in\mathbb{Z}}2^{O\left(k+i,k-i,z\right)}GC\left(k+i,k-i,z\right)q_{2}\left(k+i,k-i,z\right)\\
 &  & \quad+2^{-\#E}\sum_{k\geqslant0}\sum_{z\geqslant0}\sum_{i\in\mathbb{Z}}2^{O\left(k+i+1,k-i,z\right)}GC\left(k+i+1,k-i,z\right)q_{2}\left(k+i+1,k-i,z\right)
\end{eqnarray*}
For all $k$ and $z$, because of the symmetry of the function $GC$
and the function $O$, one has:
\[
\sum_{i\in\mathbb{Z}}2^{O\left(k+i,k-i,z\right)}GC\left(k+i,k-i,z\right)\left(q_{1}\left(k+i,k-i,z\right)-q_{2}\left(k+i,k-i,z\right)\right)
\]
\[
=\sum_{i\geqslant0}2^{O\left(k+i,k-i,z\right)}GC\left(k+i,k-i,z\right)q\left(k+i,k-i,z\right)
\]
Recall that $q\left(k+i,k-i,z\right)$ might be negative, see $\left(\ref{eq:qnegatif}\right)$,
and because of lemma $\ref{lem:LemmeClef1}$ for all $k$, there exists
an $i_{0}$ such that for all $0\leqslant i\leqslant i_{0}$:
\begin{eqnarray*}
GC\left(k+i,k-i,z\right) & \leqslant & GC\left(k+i_{0},k-i_{0},z\right)\\
GC\left(k+i+1,k-i,z\right) & \leqslant & GC\left(k+i_{0}+1,k-i_{0},z\right)\\
O\left(k+i,k-i,z\right) & \leqslant & O\left(k+i_{0},k-i_{0},z\right)\\
q\left(k+i,k-i,z\right) & \leqslant & 0
\end{eqnarray*}
and for all $i\geqslant i_{0}$:
\begin{eqnarray*}
GC\left(k+i,k-i,z\right) & \geqslant & GC\left(k+i_{0},k-i_{0},z\right)\\
GC\left(k+i+1,k-i,z\right) & \geqslant & GC\left(k+i_{0}+1,k-i_{0},z\right)\\
O\left(k+i,k-i,z\right) & \geqslant & O\left(k+i_{0},k-i_{0},z\right)\\
q\left(k+i,k-i,z\right) & \geqslant & 0
\end{eqnarray*}
By lemma $\ref{lem:LemmeClef2_1}$:
\begin{align*}
 & \sum_{i\geqslant0}2^{O\left(k+i,k-i,z\right)}GC\left(k+i,k-i,z\right)q\left(k+i,k-i,z\right)\\
 & \qquad\geqslant2^{O\left(k+i_{0},k-i_{0},z\right)}GC\left(k+i_{0},k-i_{0},z\right)\sum_{i=z}^{k-z}q\left(k+i,k-i,z\right)
\end{align*}
And by lemma $\ref{lem:LemmeClef2_2}$:
\begin{align*}
 & \sum_{i\geqslant0}2^{O\left(k+i+1,k-i,z\right)}GC\left(k+i+1,k-i,z\right)q\left(k+i+1,k-i,z\right)\\
 & \qquad\geqslant2^{O\left(k+i_{0}+1,k-i_{0},z\right)}GC\left(k+i_{0}+1,k-i_{0},z\right)\sum_{i=z}^{k-z}q\left(k+i+1,k-i,z\right)
\end{align*}
This concludes the proof since:
\begin{align*}
 & \mathbb{P}_{1/2}\left(s_{0}\leftrightarrow s_{n}\right)-\mathbb{P}_{1/2}\left(s_{0}\leftrightarrow s_{2n}\right)\\
 & =2^{-\#E}\sum_{k\geqslant0}\sum_{z\geqslant0}\sum_{i\geqslant0}2^{O\left(k+i+,k-i,z\right)}GC\left(k+i,k-i,z\right)q\left(k+i,k-i,z\right)\\
 & \qquad+2^{-\#E}\sum_{k\geqslant0}\sum_{z\geqslant0}\sum_{i\geqslant0}2^{O\left(k+i+1,k-i,z\right)}GC\left(k+i+1,k-i,z\right)q\left(k+i+1,k-i,z\right)\\
 & \geqslant0
\end{align*}
\begin{flushright} $\square$ \end{flushright}

\thanks{The author thanks Cyril Roberto, Julien Bureaux and Florent Barret
for the many insightful conversations about the topic as well as Joseba
Dalmau and Anna Bonnet.\bibliographystyle{plain}
\bibliography{bibtex}
}
\end{document}